\documentclass[12pt]{amsart}
\usepackage{amssymb,amsmath,amsthm,mathrsfs}

\newtheorem{thm}{Theorem}[section]
\newtheorem{lemma}[thm]{Lemma}
\newtheorem{cor}[thm]{Corollary}

\numberwithin{equation}{section}
\theoremstyle{definition}

\newtheorem{ex}[thm]{Example}

\newtheorem*{remark}{Remark}

\newcommand{\Cay}{\mathrm{Cay}}

\newcommand{\ca}{\mathcal{A}}

\newcommand{\bbZ}{\mathbb{Z}}

\newcommand{\la}{\langle}
\newcommand{\ra}{\rangle}

\begin{document}
\title[Nathanson's Heights and the CSS Conjecture for Cayley Graphs]
{Nathanson's Heights and the CSS Conjecture \\ for Cayley Graphs}
\author{Yotsanan Meemark and Chaiwat Pinthubthaworn}
\address{Yotsanan Meemark\\ Department of Mathematics\\ Faculty of Science
\\ Chulalongkorn University\\ Bangkok, 10330 THAILAND}
\email{\tt yotsanan.m@chula.ac.th}

\address{Chaiwat Pinthubthaworn\\ Department of Mathematics\\ Faculty of Science
\\ Chulalongkorn University\\ Bangkok, 10330 THAILAND}
\email{\tt kruchaiwat@hotmail.com}



\begin{abstract}
Let $G$ be a finite directed graph, $\beta(G)$ the minimum size of a subset $X$ of edges such that the graph $G' = (V,E \smallsetminus X)$ is directed acyclic and $\gamma(G)$  the number of pairs of nonadjacent vertices in the undirected graph obtained from $G$ by replacing each directed edge with an undirected edge. Chudnovsky, Seymour and Sullivan  \cite{CSS07} proved that if $G$ is triangle-free, then $\beta(G) \leq \gamma(G)$. They conjectured a sharper bound (so called the ``CSS conjecture") that $\beta(G) \leq \dfrac{\gamma(G)}{2}$.  Nathanson and Sullivan verified this conjecture for the directed Cayley graph $\Cay(\bbZ/N\bbZ, E_A)$ whose vertex set is the additive group $\bbZ/N\bbZ$ and whose edge set $E_A$ is determined by $E_A = \left\{(x,x+a) : x \in \bbZ/N\bbZ, a \in A\right\}$ when $N$ is prime in \cite{NS07} by introducing ``height". In this work, we extend the definition of height and the proof of CSS conjecture for $\Cay(\bbZ/N\bbZ, E_A)$ to any positive integer $N$.
\end{abstract}

\maketitle

\section{Introduction}

A {\it finite directed graph} $G=(V,E)$ consists of two finite sets, the set $V=V(G)$ of {\it vertices} of $G$ and the set $E = E(G) \subseteq V \times V$ of {\it edges} of $G$. Let $v$ and $v'$ be distinct vertices of the finite directed graph $G$. A {\it directed path of length $l$} in $G$ from $v$ to $v'$ is a sequence of $l$ edges $\{(v_{i-1}, v_{i})\}_{i=1}^l$ such that $v=v_0$ and $v'=v_l$. A {\it directed cycle of length~$l$} in $G$ is a sequence of $l$ edges $\{(v_{i-1}, v_{i})\}_{i=1}^l$ such that $v_0 = v_l$. A {\it loop}, a {\it digon} and a {\it triangle} are directed cycle of length $1$, $2$ and $3$, respectively. A {\it triangle free} graph is a graph with no loops, digons, or triangles. A directed graph is called {\it acyclic} if it has no directed cycles. 

Let $\beta(G)$ be the minimum size of a subset $X$ of edges such that the graph $G' = (V,E \smallsetminus X)$ is directed acyclic, and let $\gamma(G)$ be the number of pairs of nonadjacent vertices in the undirected graph obtained from $G$ by replacing each directed edge with an undirected edge.  Chudnovsky, Seymour and Sullivan \cite{CSS07} proved that if $G$ is a triangle-free digraph, then $\beta(G) \leq \gamma(G)$.
They conjectured a sharper bound (so called the ``CSS conjecture") that if $G$ is a triangle-free digraph, then $\beta(G) \leq \dfrac{\gamma(G)}{2}$. 

Let $N$ be a positive integer and $A$ a nonempty subset of $\bbZ/N\bbZ \smallsetminus \{0\}$ of cardinality $d \le N$. Consider the directed Cayley graph $G = \Cay(\bbZ/N\bbZ, E_A)$ whose vertex set is the additive group $\bbZ/N\bbZ$ and whose edge set $E_A$ is determined by
\[
E_A = \left\{(x,x+a) : x \in \bbZ/N\bbZ, a \in A\right\}.
\]
Assume that $G$ is triangle free. Then $G$ has neither loops nor digons, so the number of pairs of adjacent vertices is the same as the number of directed edges, which is $dN$. Thus the number of pairs of nonadjacent vertices is
\begin{equation} \label{gammaG}
\gamma(G) = \binom{N}{2} - dN = \dfrac{N(N-1-2d)}{2}. 
\end{equation}
In this case, the inequality in the CSS conjecture becomes
\[
\beta(G) \le \dfrac{\gamma(G)}{2} = \dfrac{N(N-1-2d)}{4}.
\]
By introducing the term ``height in finite projective space", Nathanson and Sullivan verified this conjecture when $N$ is prime in \cite{NS07}. Later, the height on the finite projective line was studied extensively in \cite{N08}.

Using the ``height" idea together with some elementary number theory facts involving the unit group of $\bbZ/N\bbZ$ and its cardinality, we prove the CSS conjecture when $N$ is any positive integer expanding Nathanson and Sullivan's results. The detail of our work is divided into two sections. Section 2 presents the definition and bound of the height defined for $\bbZ/N\bbZ$. The final section talks about the CSS conjecture and shows how to relate the height to it.

This work grows out of the second author's master thesis at Chulalongkorn university written under the direction of the first author to which the second author expresses his gratitude. 

\bigskip

\section{Heights}

Let $N$ and $d$ be positive integers. 
We define an equivalence relation $\sim$ on the set of nonzero $d$-tuple $(\bbZ/N\bbZ)^d\smallsetminus {(0,\dots,0)}$ by
\[
(a_1, a_2, \dots, a_d) \sim (b_1, b_2, \dots, b_d) 
\Leftrightarrow (b_1, b_2, \dots, b_d) = \lambda(a_1, a_2, \dots, a_d)  
\]
for some $\lambda \in (\bbZ/N\bbZ)^\times$. Here $(\bbZ/N\bbZ)^\times$ stands for the unit group of $\bbZ/N\bbZ$ and we use $(\bbZ/N\bbZ)^*$ for the set of nonzero element in $\bbZ/N\bbZ$. Observe that $(\bbZ/N\bbZ)^\times = (\bbZ/N\bbZ)^*$ if and only if $N$ is a prime. Also, $|(\bbZ/N\bbZ)^\times| = \phi(N)$, the {\it Euler $\phi$-function}. Write $(a \mod{N})$ for the least nonnegative integer in the congruence class $a \in \bbZ/N\bbZ$. We first compute

\begin{lemma} \label{sumka}
For $a \in (\bbZ/N\bbZ)^*$, $$ \sum_{k \in (\bbZ/N\bbZ)^\times}(ka \mod{N}) = \frac{N\phi(N)}{2}.$$
\end{lemma} 
\begin{proof}
Let $a \in (\bbZ/N\bbZ)^*$. If $N = 2$, then $(a \mod{2}) = 1 = \dfrac{2\phi(2)}{2}$. Next we assume that  $N > 2$.
It is clear that  $k \in (\bbZ/N\bbZ)^\times \Leftrightarrow N-k \in (\bbZ/N\bbZ)^\times$ for all $k \in (\bbZ/N\bbZ)^*$.
Since $N > 2$, $k \neq N-k$ for every $k \in (\bbZ/N\bbZ)^\times$. Then 
\[
(\bbZ/N\bbZ)^\times = \left\{ k,N-k : k \in (\bbZ/N\bbZ)^\times \; \text{and} \; k < \frac{N}{2} \right\}
\]
and so $\phi(N)$ is even. Note that 
\[
((N-k)a \mod{N}) = ((Na-ka)\mod{N}) = N-(ka\mod{N})
\]
for all $k \in (\bbZ/N\bbZ)^\times$. Thus
\begin{align*}
\sum_{k \in (\bbZ/N\bbZ)^\times}  (ka \mod{N}) 
&= \sum_{\substack{k \in (\bbZ/N\bbZ)^\times, \\ k < N/2}} [(ka \mod{N})+((N-k)a\mod{N})] \\
&= \sum_{\substack{k \in (\bbZ/N\bbZ)^\times, \\ k < N/2}} [(ka \mod{N})+(N-(ka\mod{N}))] \\
&= \sum_{\substack{k \in (\bbZ/N\bbZ)^\times, \\ k < N/2}} N = \frac{N\phi(N)}{2}.
\end{align*}
Hence we have the lemma.
\end{proof} 

We denote the equivalence class of the point $(a_1,a_2,\dots,a_d)$ by $\left\langle a_1,a_2,\dots,a_d \right\rangle$ and the set of all equivalence classes by $\mathbf{P}^{d-1}(\bbZ/N\bbZ)$. The  {\it{height}} of the class ${\bf{a}} = \left\langle a_1,a_2,\dots,a_d \right\rangle \in \mathbf{P}^{d-1}(\bbZ/N\bbZ)$ is given by
\[
h_N ({\bf{a}}) = \min \left\{  \sum_{i=1}^{d}(k a_i \mod N) : k \in (\bbZ/N\bbZ)^\times  \right\} .
\]
Since ${\bf{a}} \neq {\bf{0}}$, there exists $a_j \in (\bbZ/N\bbZ)^*$ such that $(k{a_j} \mod{N}) > 0$ for every $k \in (\bbZ/N\bbZ)^\times$, so $h_N: \mathbf{P}^{d-1}(\bbZ/N\bbZ) \rightarrow \mathbb{Z}^+$.
We use $d^*({\bf{a}})$ to denote the number of nonzero components of  ${\bf{a}} = \left\langle a_1,\dots,a_d \right\rangle \in \mathbf{P}^{d-1}(\bbZ/N\bbZ)$, that is, the number of $a_i \neq 0$, and we define $$d^*({\ca}) = \text{max}\left\{ d^*({\bf{a}}) : {\bf{a}} \in \ca \right\}$$ for  $\ca \subseteq \mathbf{P}^{d-1}(\bbZ/N\bbZ)$. Clearly, $h_N ({\bf{a}}) \leq d^*({\bf{a}})(N-1)$ for all ${\bf{a}} \in \mathbf{P}^{d-1}(\bbZ/N\bbZ)$. For any nonempty finite subset $A$ of $\bbZ^+$ with $\left|A\right| = m$, we note that $\min A \leq \dfrac{1}{m}\sum_{a \in A} a$. By Lemma \ref{sumka}, we have
\begin{align*}
h_N({\bf{a}}) &= \text{min} \left\{ \sum_{i=1}^{d}(ka_i \mod{N}) : k \in (\bbZ/N\bbZ)^\times  \right\}\\
&\leq \frac{1}{\phi(N)}\left(\sum_{k \in (\bbZ/N\bbZ)^\times}\left(\sum_{i=1}^{d}(ka_i \mod{N})\right)\right)\\
&= \frac{1}{\phi(N)}\left(\sum_{i=1}^{d}\left(\sum_{k \in (\bbZ/N\bbZ)^\times}(ka_i \mod{N})\right)\right)\\
&= \frac{1}{\phi(N)}\left(d^*({\bf{a}})\frac{N\phi(N)}{2}\right) = \frac{d^*({\bf{a}})N}{2}.
\end{align*}
Since the heights are the positive integers, $h_N ({\bf{a}}) \leq \left\lfloor \dfrac{d^*({\bf{a}})N}{2}\right\rfloor$.
Hence we get a better bound for $h_N ({\bf{a}})$. We summarize the above computation with its corollary as follows.

\begin{lemma} \label{ubhn}
For ${\bf{a}} \in \mathbf{P}^{d-1}(\bbZ/N\bbZ), h_N ({\bf{a}}) \leq \left\lfloor \dfrac{d^*({\bf{a}})N}{2}\right\rfloor$.
\end{lemma} 

\begin{cor} \label{ubhnc}
{\rm (i)} For $d \geq 1$ and ${\bf{a}} \in \mathbf{P}^{d-1}(\bbZ/2\bbZ)$, $h_2 ({\bf{a}}) = d^*({\bf{a}})$. \\
{\rm (ii)} For $N \geq 2$ and ${\bf{a}} = \left\langle a \right\rangle \in \mathbf{P}^{0}(\bbZ/N\bbZ), h_N({\bf{a}}) \leq \left\lfloor \dfrac{N}{2} \right\rfloor$. In particular, if $a \in (\bbZ/N\bbZ)^\times$, then $h_N({\bf{a}}) = \min \left\{ka \mod{N} : k \in (\bbZ/N\bbZ)^\times \right\} = \min \left\{ (k \mod{N})  : k \in (\bbZ/N\bbZ)^\times \right\}= 1$. 
\end{cor}

\bigskip

\section{The CSS Conjecture}

In this section, we deal with the CSS conjecture for the Cayley graph $G = \Cay(\bbZ/N\bbZ, E_A)$. Notice that if the outdegree of every vertex in finite directed graph is at least one, then the graph contains a cycle. Thus every finite directed acyclic graph contain at least one vertex with outdegree 0.  Nathanson and Sullivan used this to prove the following theorem and derived its consequence. Their proofs can be found in \cite{NS07}. We recall this work in


\begin{thm} 
\cite{NS07}\label{acyclic} 
Let $V = \left\{v_0, v_1, \dots ,v_{N-1} \right\}$ be  the vertex set of the directed graph $G$. 
Then $G$ is directed acyclic if and only if there is a permutation $\sigma$ of $\{0, 1, \dots, N-1 \}$ such that 
$r < s$ for every edge $(v_{\sigma(r)},v_{\sigma(s)})$ of the graph $G$.
\end{thm} 
 
\begin{cor} 
\cite{NS07}\label{bdheight} 
Let $G = (V,E)$ be a directed graph with vertex set $\{v_0, v_1, \dots, v_{N-1} \}$ and let $\Sigma \subseteq S_N$ be a set of permutations of $\{0, 1, \dots, N-1 \}$.
For $\sigma \in \Sigma$, let $B_{\sigma}$ be the set of edges $(v_{\sigma(r)},v_{\sigma(s)}) \in E$ with 
$r \ge s$. Then $\beta(G) \leq \min\left\{\left|B_{\sigma}\right| : \sigma \in \Sigma \right \}$.
\end{cor}

This corollary yields an immediate result on our Cayley graph $\Cay(\bbZ/N\bbZ,E_A)$, namely,

\begin{lemma} \label{cayley}
Let $N \geq 2$, $d \ge 1$ and $A = \left\{a_1,\dots,a_d\right\} \subseteq (\bbZ/N\bbZ)^*$. 
Let $G = \Cay(\bbZ/N\bbZ,E_A)$ be the Cayley graph constructed from $A$.
Let $\Sigma$ be a set of permutations of $\bbZ/N\bbZ$ and $\sigma \in \Sigma$.
For $i \in \bbZ/N\bbZ$ and $j \in \left\{1,\dots,d\right\}$, define $t_{ij} \in \bbZ/N\bbZ$ by 
$\sigma(i)+a_j = \sigma(t_{ij})$. Then 
\[
E_A = \left\{(\sigma(i),\sigma(t_{ij})) : i \in \bbZ/N\bbZ \; \; \text{and} \;\; j \in \left\{1,\dots,d\right\}\right\}.
\]
Let 
\[
B_{\sigma} = \left\{(\sigma(i),\sigma(t_{ij})) : (i \mod{N}) > (t_{ij} \mod{N}) \; \; \text{and} \;\; j \in \left\{1,\dots,d\right\} \right\}. 
\]
Then the graph $G' = (\bbZ/N\bbZ,E_A \smallsetminus B_{\sigma})$ is directed acyclic for every permutation $\sigma \in \Sigma$ and 
\[
\beta(G) \leq \min\left\{\left|B_{\sigma}\right| : \sigma \in \Sigma \right\}.
\] 
\end{lemma} 

For $k \in (\bbZ/N\bbZ)^\times$, define the permutation $\sigma_k$ of $\bbZ/N\bbZ$ by $\sigma_k(i) = ki$ for all $i \in \bbZ/N\bbZ$. Consider the set $\Sigma = \left\{\sigma_k : k \in (\bbZ/N\bbZ)^\times\right\}$  of $\phi(N)$ permutations of $\bbZ/N\bbZ$. Fix $k \in (\bbZ/N\bbZ)^\times$. 
For $i \in \bbZ/N\bbZ$ and $j \in \left\{1,\dots,d\right\}$, define
$t_{ij} \in \bbZ/N\bbZ \smallsetminus \{ i \}$ by $\sigma_k(t_{ij}) = \sigma_k(i) + a_j$.
Since $k \in (\bbZ/N\bbZ)^\times$, there exists $u_k \in (\bbZ/N\bbZ)^\times$ such that 
$ku_k = 1$.
Let $r_j = (u_ka_j \mod{N})$. Then $1 \leq r_j \leq N-1$ and $a_j = kr_j$. Thus
\[
\sigma_k(t_{ij}) = \sigma_k(i)+a_j = ki + kr_j = k(i + r_j) = \sigma_k(i+r_j),
\]
so  $t_{ij} = i+r_j$. 
Since $1 \leq r_j \leq N-1$, $(t_{ij}\mod N) = (i~\mod N)+r_j-N < (i~\mod N)$ if $(i~\mod N)+r_j \geq N$. 
Moreover, if $(i~\mod N)+r_j < N$, then $(t_{ij}\mod N) = (i~\mod N)+r_j > (i~\mod N)$. 
Hence $(i~\mod N) > (t_{ij}\mod N) \Leftrightarrow N-r_j \leq (i~\mod N)$. 

Let $B_{\sigma_k} = \left\{(\sigma_k(i),\sigma_k(t_{ij})) : (i \mod N) > (t_{ij} \mod N) 
\; \; \text{and} \;\; j \in \left\{1,\dots,d\right\} \right\}$.
Then
\[
\left|B_{\sigma_k}\right| 
= \left|\left\{(\sigma_k(i),\sigma_k(t_{ij})) : N-r_j \leq (i \mod N) \leq N-1\right\}\right| 
= \sum_{j=1}^{d}r_j = \sum_{j=1}^{d}(u_ka_j \mod{N}).
\]
Applying Lemma \ref{cayley} and the fact that $\left\{u_k : k \in (\bbZ/N\bbZ)^\times\right\} = (\bbZ/N\bbZ)^\times$, we get
\begin{align*}
\beta(G) &\leq \min\left\{\left|B_{\sigma_k}\right| : k \in (\bbZ/N\bbZ)^\times\right\}\\
&= \min\left\{\sum_{j=1}^{d}(u_ka_j \mod{N}) : k \in (\bbZ/N\bbZ)^\times\right\}\\
&= \min\left\{\sum_{j=1}^{d}(ka_j \mod{N}) : k \in (\bbZ/N\bbZ)^\times\right\}\\
&= h_N(\left\langle a_1,\dots,a_d\right\rangle).
\end{align*}
Thus $\beta(G) \le h_N(\left\langle a_1,\dots,a_d\right\rangle)$. Together with Lemma \ref{ubhn}, we have

\begin{lemma} \label{betahn}
Let $N \geq 2$, $d \ge 1$ and $A = \left\{a_1,\dots,a_d\right\} \subseteq (\bbZ/N\bbZ)^*$. 
Let $G = \Cay(\bbZ/N\bbZ,E_A)$ be the Cayley graph constructed from $A$.
Then $$\beta(G) \leq h_N(\left\langle a_1,\dots,a_d\right\rangle) \leq \frac{dN}{2}.$$
\end{lemma} 

This lemma gives

\begin{thm} \label{css}
Let $N \geq 5, d \geq 1$ and $A = \left\{a_1,\dots,a_d\right\} \subseteq (\bbZ/N\bbZ)^*$. 
Let $G = \Cay(\bbZ/N\bbZ,E_A)$ be the Cayley graph constructed from $A$ which has no digons.
If $d \leq \dfrac{N-1}{4}$, then $\beta(G) \leq \dfrac{\gamma(G)}{2}$.
\end{thm} 
\begin{proof}
Assume that $d \leq \dfrac{N-1}{4}$. Then 
\[
\dfrac{dN}{2} = dN-\dfrac{dN}{2} \leq \dfrac{N(N-1)}{4}-\dfrac{dN}{2} = \dfrac{N(N-1-2d)}{4}.
\]
By Lemma \ref{betahn} and Eq. \eqref{gammaG}, we get $\beta(G) \leq \dfrac{dN}{2} \leq \dfrac{N(N-1-2d)}{4} = \dfrac{\gamma(G)}{2}$.
\end{proof}

Hamidoune proved the Caccetta-H\"{a}ggkvist conjecture for Cayley graphs:

\begin{thm} \label{Hami}
{\rm \cite{H81,N06}} Let $A \subseteq (\bbZ/N\bbZ)^*$ and $d = \left|A\right| \geq \dfrac{N}{k}$.
Then the Cayley graph $G = \Cay(\bbZ/N\bbZ,E_A)$ contains a cycle of length at most $k$.
In particular, if $G$ is triangle-free, then $d < \dfrac{N}{3}$, that is, $3d+1 \leq N$. 
\end{thm}

From Theorem \ref{css}, if $N \geq 4d+1$, then $\beta(G) \leq \dfrac{\gamma(G)}{2}$.
Combined with Theorem \ref{Hami}, in order to prove the CSS conjecture for the triangle-free Cayley graph $G$ with $A$ of size $d$,
it suffices to consider only when $3d+1 \leq N \leq 4d$.

For $A \subseteq \bbZ/N\bbZ$ and $l$ is a positive integer, we define 
\[
lA = \underbrace{A + A + \dots + A}_\text{$l$ copies} = \left\{a_1+a_2+\dots+a_l : a_i \in A \;\text{for}\; i = 1,2,\dots,l\right\}.
\]
We have a criterion for determining whether our Cayley graph $\Cay(\bbZ/N\bbZ,E_A)$ is triangle-free in the next lemma.

\begin{lemma} 
\cite{N06}\label{cycle} 
Let $A \subseteq \bbZ/N\bbZ$ and $G = \Cay(\bbZ/N\bbZ,E_A)$ be the Cayley graph constructed from $A$. 
Then $G$ contains a directed cycle of length $l$ if and only if $0 \in lA$.
In particular, $G$ is a triangle-free digraph if and only if $0 \notin A, 0 \notin 2A$ and $0 \notin 3A$.
\end{lemma}

This lemma allows us to give results for some small $d$ as follows. 

\begin{thm}
Let $N \ge 4$ and $A \subseteq (\bbZ/N\bbZ)^*$. Then the CSS conjecture holds for the Cayley graphs $G = \Cay(\bbZ/N\bbZ,E_A)$ when $|A| = 1, 2$ or $3$. That is,  \\
{\rm(i)} $A = \left\{a_1\right\}$, \qquad
{\rm(ii)} $A = \left\{a_1,a_2\right\}$, \qquad
{\rm(iii)} $A = \left\{a_1,a_2,a_3\right\}$.
\end{thm}
\begin{proof}
(i) It suffices to consider only for $N = 4$. 
If $a_1 = 2$, then $0 \in 2A$, so by Lemma \ref{cycle}, $G$ has a digon, a contradiction.
Now assume $a_1 = 1$ or $3$. Then $a_1 \in (\bbZ/4\bbZ)^\times$.
By Corollary \ref{ubhnc} (ii), $h_4(\left\langle a_1 \right\rangle) = 1$, 
so $\beta(G) \leq h_4(\left\langle a_1 \right\rangle) = 1 = \dfrac{\gamma(G)}{2}$. \\ 

\noindent (ii) It suffices to consider only when $N = 7$ or $8$. The following table displays the heights $h_N(\la a_1, a_2 \ra)$ for $N = 7$ and $N=8$.
\[
\begin{tabular}{|c|c|c||c|c|c||c|c|c|}
\hline
$N$ & $\left\langle a_1,a_2\right\rangle$ & $h_N(\left\langle a_1,a_2\right\rangle)$ &
$N$ & $\left\langle a_1,a_2\right\rangle$ & $h_N(\left\langle a_1,a_2\right\rangle)$ &
$N$ & $\left\langle a_1,a_2\right\rangle$ & $h_N(\left\langle a_1,a_2\right\rangle)$ \\\hline
$7$ & $\left\langle 1,2\right\rangle^*$ & $3$ &
$8$ & $\left\langle 1,2\right\rangle^*$ & $3$ &
$8$ & $\left\langle 2,3\right\rangle$ & $5$ \\\hline
 & $\left\langle 1,3\right\rangle$ & $4$ &
 & $\left\langle 1,3\right\rangle^*$ & $4$ &
 & $\left\langle 2,4\right\rangle$ & $6$ \\\hline
 & $\left\langle 1,4\right\rangle^*$ & $3$ &
 & $\left\langle 1,4\right\rangle$ & $5$ &
 & $\left\langle 2,5\right\rangle^*$ & $3$ \\\hline
 & $\left\langle 1,5\right\rangle$ & $4$ &
 & $\left\langle 1,5\right\rangle^*$ & $6$ &
 & $\left\langle 2,6\right\rangle$ & $8$ \\\hline
 & $\left\langle 1,6\right\rangle$ & $7$ &
 & $\left\langle 1,6\right\rangle$ & $5$ &
 & $\left\langle 4,5\right\rangle$ & $5$ \\\hline
 &  &  &
 & $\left\langle 1,7\right\rangle$ & $8$ &
 & $\left\langle 4,6\right\rangle$ & $6$ \\\hline
\end{tabular}
\]
Here $\left\langle a_1,a_2\right\rangle^*$ means that $G = \Cay(\bbZ/N\bbZ,E_A)$
with $A = \left\{a_1,a_2\right\}$ is triangle-free.
Otherwise the graph $G$ is not triangle-free. For instance, 
when $N = 8$ and $A = \left\{1,6\right\}$. Since $0 \in 3A = \left\{0,2,3,5\right\}$, $G = \Cay(\bbZ/N\bbZ,E_A)$ contains a cycle of length $3$ by Lemma \ref{cycle}.\\
If $N = 7$, then $\beta(G) \leq h_7(\left\langle a_1,a_2\right\rangle^*) = 3 < \dfrac{7(7-1-4)}{4} = \dfrac{\gamma(G)}{2}$.\\
If $N = 8$, then $\beta(G) \leq h_8(\left\langle a_1,a_2\right\rangle^*) \leq 6 = \dfrac{8(8-1-4)}{4} = \dfrac{\gamma(G)}{2}$. \\ 

\noindent (iii) It suffices to verify this conjecture only when $10 \leq N \leq 12$.
A similar computation as in the previous case reduces the table of $h_N(\left\langle a_1,a_2,a_3\right\rangle)$ for $10 \leq N \leq 12$ with $0<(a_1\mod{N})<(a_2\mod{N})<(a_3\mod{N})$ and $G = \Cay(\bbZ/N\bbZ,E_{\{a_1,a_2,a_3\}})$ is triangle-free as follows:
\[
\begin{tabular}{|c|c|c||c|c|c|}
\hline
$N$ & $\left\langle a_1,a_2,a_3\right\rangle$ & $h_N(\left\langle a_1,a_2,a_3\right\rangle)$ & $N$ & $\left\langle a_1,a_2,a_3\right\rangle$ & $h_N(\left\langle a_1,a_2,a_3\right\rangle)$ \\\hline
$10$ & $\left\langle 1,2,3\right\rangle$ & $6$ & $12$ & $\left\langle 1,2,3\right\rangle$ & $6$
\\\hline
 & $\left\langle 1,4,7\right\rangle$ & $6$ & & $\left\langle 1,2,7\right\rangle$ & $10$  
\\\hline
$11$ & $\left\langle 1,2,3\right\rangle$ & $6$ & & $\left\langle 1,3,5\right\rangle$ & $9$ 
\\\hline
 & $\left\langle 1,2,4\right\rangle$ & $7$ & & $\left\langle 1,3,7\right\rangle$ & $11$ 
\\\hline
 & $\left\langle 1,2,6\right\rangle$ & $7$ & & $\left\langle 1,5,9\right\rangle$ & $15$ 
\\\hline
 & $\left\langle 1,3,6\right\rangle$ & $7$ & & $\left\langle 1,7,9\right\rangle$ & $11$ 
\\\hline
 & $\left\langle 1,4,8\right\rangle$ & $6$ & & &
\\\hline
 & $\left\langle 1,6,7\right\rangle$ & $6$ & & &
\\\hline
\end{tabular}
\]
If $N = 10$, then $\beta(G) \leq h_{10}(\left\langle a_1,a_2,a_3\right\rangle) = 6 < \dfrac{10(10-1-6)}{4} = \dfrac{\gamma(G)}{2}$.\\
If $N = 11$, then $\beta(G) \leq h_{11}(\left\langle a_1,a_2,a_3\right\rangle) \leq 7 < \dfrac{11(11-1-6)}{4} = \dfrac{\gamma(G)}{2}$.\\
If $N = 12$, then $\beta(G) \leq h_{12}(\left\langle a_1,a_2,a_3\right\rangle) \leq 15 = \dfrac{12(12-1-6)}{4} = \dfrac{\gamma(G)}{2}$.
\end{proof}

\begin{remark}
For $d>3$, the following example shows that sometimes the height is greater than $\dfrac{\gamma(G)}{2}$,
so we cannot conclude the CSS conjecture without computing $\beta(G)$ explicitly.
\end{remark}

\begin{ex}
Let $N = 14$ and $A = \left\{1,2,8,9\right\} \subset (\bbZ/14\bbZ)^*$.
Since $0$ is not in $A, 2A$ and $3A$, $G = \Cay(\bbZ/14\bbZ,E_A)$ is a triangle-free digraph by Lemma \ref{cycle}.
We have $h_{14} (\left\langle 1,2,8,9\right\rangle) = 20$ and $\gamma(G) = 35$.
Thus $h_{14} (\left\langle 1,2,8,9\right\rangle) > \dfrac{\gamma(G)}{2}$.
\end{ex}

\end{document}